\documentclass[12pt]{article}

\usepackage{amssymb}%
\usepackage{amsmath}%
\usepackage{amsthm}%

\usepackage{xcolor}%

\allowdisplaybreaks%

{\theoremstyle{plain}%
 \newtheorem{theorem}{Theorem}

}
{\theoremstyle{remark}

}
{\theoremstyle{definition}

\newtheorem{example}{Example}
}

\begin{document}

\begin{center}
{\large Three-parameter generalizations of formulas due to Guillera}

 \ 

{\textsc{John M. Campbell}} 

 \ 

\end{center}

\begin{abstract}
 Guillera has introduced remarkable series expansions for $\frac{1}{\pi^2}$ of convergence rates $-\frac{1}{1024}$ and $-\frac{1}{4}$ via the 
 Wilf--Zeilberger method. Through an acceleration method based on Zeilberger's algorithm and related to Chu and Zhang's series 
 accelerations based on Dougall's ${}_{5}H_{5}$-series, we introduce and prove three-parameter generalizations of Guillera's formulas. 
 We apply our method to construct rational, hypergeometric series for $\frac{1}{\pi^2}$ that are of the same convergence rates as 
 Guillera's series and that have not previously been known. 
\end{abstract}


\vspace{0.1in}


\section{Introduction}
 A groundbreaking development in the areas of mathematics influenced by Ramanujan is given by Guillera's discoveries on series for $ 
 \frac{1}{\pi^2}$ \cite{Guillera2003,Guillera2006,Guillera2008,Guillera2011,Guillera2002}. This is evidenced by the extent of the influence of 
 Guillera's formulas on subsequent work related to series for $\frac{1}{\pi^2}$; see 
 \cite{Almkvist2009,BaruahBerndt2009,ChenChu2021q,Chu2023,Chu2011,Chu2018,Chu2021Infinite,ChuZhang2014,He2016,LiChu2024,
WeiWangDai2015,Zudilin2007More,Zudilin2007Quadratic} and many related references. While Ramanujan's series for $\frac{1}{\pi}$ 
 \cite[pp.\ 352--354]{Berndt1994} \cite{Ramanujan1914} are well understood to be closely related to modular forms and the 
 behavior of elliptic integrals and theta functions \cite{BorweinBorwein1987}, the theory underlying Guillera's series for 
 $\frac{1}{\pi^2} $ is often seen as a relative mystery \cite[\S7.2]{BaileyBorweinCalkinGirgensohnLukeMoll2007} \cite{Guillera2011}, due 
 to the ``black box'' nature of the Wilf--Zeilberger (WZ) method 
 \cite{PetkovsekWilfZeilberger1996} employed in the derivation of such series. This motivates the development of techniques that 
 could give light to Guillera's formulas for $\frac{1}{\pi^2}$ and to how such formulas may be extended. This provides the main purpose 
 of our paper. 

 Let the $\Gamma$-function be defined by an Euler integral so that $\Gamma(x)$ $:=$ $ \int_{0}^{\infty} u^{x-1} $ $ e^{-u} $ $ du$ 
 for $\Re(x) > 0$, with the analytic continuation for $\Gamma$ extending the definition to complex numbers apart from negative 
 integers, via the relation $\Gamma(z) = \frac{ \Gamma(z+1) }{z}$. We may then let the Pochhammer symbol be defined so that 
 $ (a)_k := \frac{\Gamma (a+k)}{\Gamma (a)}$. In this direction, we are to make use of the notational convention 
\begin{equation*}
 \left[ \begin{matrix} \alpha, \beta, \ldots, \gamma \vspace{1mm} \\ 
 A, B, \ldots, C \end{matrix} \right]_{k} = \frac{ (\alpha)_{k} (\beta)_{k} 
 \cdots (\gamma)_{k} }{ (A)_{k} (B)_{k} \cdots (C)_{k}}. 
\end{equation*}
 Out of Ramanujan's 17 series for $\frac{1}{\pi}$, the simplest two such series (in terms of having a slower absolute convergence rate) 
 are such that 
\begin{align}
 \frac{4}{\pi} & = \sum_{k=0}^{\infty} 
 \left( \frac{1}{4} \right)^{k} 
 \left[ \begin{matrix} 
 \frac{1}{2}, \frac{1}{2}, \frac{1}{2} \vspace{1mm} \\ 
 1, 1, 1 \end{matrix} \right]_{k} 
 (6k+1) \ \text{and} \label{Ramanujan1} \\
 \frac{8}{\pi}
 & = \sum_{k=0}^{\infty} \left( -\frac{1}{4} \right)^{k}
 \left[ \begin{matrix} 
 \frac{1}{4}, \frac{1}{2}, \frac{3}{4} \vspace{1mm} \\ 
 1, 1, 1 \end{matrix} \right]_{k} 
 (20 k + 3). \label{Ramanujan2}
\end{align}
 This leads us to consider how rational, hypergeometric, and fast converging series for higher powers of $\frac{1}{\pi}$ 
 could be obtained, in relation to the known derivations of \eqref{Ramanujan1}--\eqref{Ramanujan2} or otherwise. 

 It was shown by Glaisher in 1905 \cite{Glaisher1905} that 
\begin{equation}\label{Glaisherseries}
 \frac{8}{\pi^2} 
 = \sum_{k=0}^{\infty}
 \left[ \begin{matrix} 
 -\frac{1}{2}, -\frac{1}{2}, -\frac{1}{2}, -\frac{1}{2} \vspace{1mm} \\ 
 1, 1, 1, 1 \end{matrix} \right]_{k} 
 (1 - 4k), 
\end{equation}
 noting the convergence rate of $1$ in \eqref{Glaisherseries}. The first known hypergeometric series for $\frac{1}{\pi^2}$ with 
 algebraic summands and absolute convergence rates strictly less than that of Glaisher's series in \eqref{Glaisherseries} were 
 introduced by Guillera in 2002 \cite{Guillera2002} and are such that 
\begin{align}
 \frac{32}{\pi^2} 
 & = \sum_{k = 0}^{\infty} 
 \left( \frac{1}{16} \right)^{k} 
 \left[ \begin{matrix} 
 \frac{1}{4}, \frac{1}{2}, \frac{1}{2}, \frac{1}{2}, \frac{3}{4} \vspace{1mm} \\ 
 1, 1, 1, 1, 1 \end{matrix} \right]_{k} 
 \left( 120 k^2 + 34 k + 3 \right) \ \text{and} \label{firstknown} \\ 
 \frac{128}{\pi^2}
 & = \sum_{k = 0}^{\infty} \left( -\frac{1}{1024} \right)^{k} 
 \left[ \begin{matrix} 
 \frac{1}{2}, \frac{1}{2}, \frac{1}{2}, \frac{1}{2}, \frac{1}{2} \vspace{1mm} \\ 
 1, 1, 1, 1, 1 \end{matrix} \right]_{k} 
 \left( 820 k^2 + 180 k + 13 \right). \label{secondknown}
\end{align}
 In contrast to the classical series due to Glaisher in \eqref{Glaisherseries}, 
 which may be evaluated according to the partial sum identity such that 
 \begin{equation*}
 \sum_{k=0}^{n} 
 \left[ \begin{matrix} 
 -\frac{1}{2}, -\frac{1}{2}, -\frac{1}{2}, -\frac{1}{2} \vspace{1mm} \\ 
 1, 1, 1, 1 \end{matrix} \right]_{k} 
 (1 - 4k) = 
 \frac{(n+1)^4 \left(8 n^2+4 n+1\right) \Gamma^4 \left(\frac{2 n+1}{2} \right)}{\pi ^2 \Gamma^{4} (n+2)}, 
\end{equation*}
 the partial sums of Guillera's series as in \eqref{firstknown}--\eqref{secondknown} cannot be evaluated in closed form. The 
 Guillera series in \eqref{firstknown}--\eqref{secondknown} were proved with the use of a WZ pair $(F, G)$, for bivariate and 
 hypergeometric functions $F = F(n, k)$ and $G = G(n, k)$, determined by Guillera via the {\tt EKHAD} package, and with the use of a 
 telescoping argument to obtain that 
\begin{equation}\label{sumGconstant}
 \sum_{n=0}^{\infty} G(n, k) = \sum_{n=0}^{\infty} G(n, k+1), 
\end{equation}
 so that an application of Carlson's theorem yields the constancy of either side of \eqref{sumGconstant} for complex $k$. This 
 approach was also applied by Guillera \cite{Guillera2006} to prove that 
\begin{equation}\label{mainGuillera}
 \frac{8}{\pi^2} 
 = \sum_{k = 0}^{\infty} 
 \left( -\frac{1}{4} \right)^{k} 
 \left[ \begin{matrix} 
 \frac{1}{2}, \frac{1}{2}, \frac{1}{2}, \frac{1}{2}, \frac{1}{2} \vspace{1mm} \\ 
 1, 1, 1, 1, 1 \end{matrix} \right]_{k} 
 \left( 20 k^2 + 8 k + 1 \right). 
\end{equation}
 Using an acceleration method related to the work of Wilf \cite{Wilf1999} and of Chu and Zhang \cite{ChuZhang2014}, we introduce 
 infinite families of generalizations of both \eqref{secondknown} and \eqref{mainGuillera}. As discussed in Section 
 \ref{sectionBackground}, our generalizations are not equivalent to the accelerations obtained by Chu and Zhang \cite{ChuZhang2014} 
 from Dougall's ${}_{5}H_{5}$-sum. We apply our method to obtain series for $\frac{1}{\pi^2}$ that have not previously been 
 known, including 
\begin{align}
 -\frac{64}{\pi^2} = & \sum_{k = 0}^{\infty} 
 \left( -\frac{1}{4} \right)^{k} 
 \left[ \begin{matrix} 
 \frac{3}{2}, \frac{3}{2}, \frac{3}{2}, \frac{3}{2}, \frac{3}{2} \vspace{1mm} \\ 
 1, 1, 1, 1, 2 \end{matrix} \right]_{k} 
 \left( 20 k^2+24 k+9 \right) \label{motivating} 
\end{align}
 and 
\begin{align}
\begin{split}
 -\frac{8}{\pi^2} 
 = & \sum_{k=0}^{\infty} 
 \left( -\frac{1}{1024} \right)^{k} 
 \left[ \begin{matrix} 
 -\frac{1}{2}, -\frac{1}{2}, -\frac{1}{2}, \frac{1}{2}, \frac{1}{2} \vspace{1mm} \\ 
 1, 1, 1, 1, 1 
 \end{matrix} \right]_{k} \times \\ 
 & \ \ \ \ \ \ \ \big( 6560 k^5-528 k^4+176 k^3+8 k^2-6 k-1 \big). 
\end{split}\label{motivating2}
\end{align}

\section{Background}\label{sectionBackground}
 Chu and Zhang's acceleration method \cite{ChuZhang2014} relies on 
 recurrences for the unilateral sum 
\begin{equation*}
 \Omega(a;b,c,d,e) := 
 \sum_{k = 0}^{\infty} 
 \left( a + 2 k \right) 
 \left[ \begin{matrix} 
 b, c, d, e \vspace{1mm} \\ 
 1 + a - b, 1 - a - c, 1 - a - d, 1 + a - e \end{matrix} \right]_{k} 
\end{equation*}
 corresponding to the well poised ${}_{5}H_{5}$-series given by Dougall \cite{Dougall1907}. Chu and Zhang introduced a 
 recurrence of the form 
\begin{equation}\label{Omegarecurrence}
 \Omega(a;b,c,d,e) 
 = R_{1} + R_{2} \, \Omega(a+2;b,c+1,d+1,e+1) 
\end{equation}
 for rational functions $R_{1}$ and $R_{2}$ \cite[Lemma 7]{ChuZhang2014}, with \eqref{Omegarecurrence} having been derived via 
 Abel's lemma on summation by parts. The iterative application of \eqref{Omegarecurrence} was then applied to prove that 
\begin{align*}
 \Omega(a;b,c,d,e)
 & = \sum_{k=0}^{\infty} \frac{(-1)^k}{ \left( 1 + a - b \right)_{2k} } 
 \alpha_{k}(a;b,c,d,e) \\ 
 & \times 
 \left[ \begin{matrix} 
 c, d, e, 1 + a - b - c, 1 + a - b - d, 1 + a - b - e \vspace{1mm} \\ 
 1 + a - c, 1 + a - d, 1 + a - e, 1 + 2a - b - c - d - e \end{matrix} \right]_{k} 
\end{align*}
 for a specified rational function $ \alpha_{k}(a;b,c,d,e) $ and for $\Re(1 + 2a - b - c - d - e) > 0$ \cite[Theorem 9]{ChuZhang2014}. It 
 was claimed by Chu and Zhang that the $a = b = c = d = e = \frac{1}{2}$ case of this result yields Guillera's formula in 
 \eqref{mainGuillera} and that the case whereby $a = b = \frac{3}{2}$ and $c = d = e = \frac{1}{2}$ yields 
\begin{equation}\label{claimedChuZhang}
 \frac{128}{\pi^2} = 
 \sum_{k=0}^{\infty} \left( -\frac{1}{4} \right)^{k} 
 \left[ \begin{matrix} 
 \frac{1}{2}, \frac{1}{2}, \frac{1}{2}, \frac{1}{2}, \frac{1}{2} \vspace{1mm} \\ 
 1, 2, 2, 2, 2 \end{matrix} \right]_{k} 
 \left( 20 k^2 + 32 k + 13 \right). 
\end{equation}
 According to the given definition for $ \Omega(a; b, c, d, e)$, the summand associated with $ \Omega(a; b, c, d, e)$ for $a = b = c = d 
 = e = \frac{1}{2}$ reduces to $ \frac{\left(2 k + \frac{1}{2} \right) \left(\frac{1}{2}\right)_k^4}{ (0)_k^2 (1)_k^2}$, but the Pochhammer 
 symbol $(0)_{k}$ vanishes for $k \in \mathbb{N}$, and similarly for the combination of values for $a$, $b$, $c$, $d$, and $e$ 
 purportedly yielding \eqref{claimedChuZhang}. This leads us to consider how the techniques of Chu and Zhang could be modified 
 so as to obtain infinite families of generalizations of \eqref{mainGuillera}. In this direction, we are to make use of an acceleration based 
 on a summand inequivalent to the summand for $\Omega(a,b,c,d,e)$ for any combination of arguments, so that our generalizations 
 of Guillera's formulas in \eqref{mainGuillera} and \eqref{secondknown} are not equivalent to the accelerations from Chu and Zhang's 
 work \cite{ChuZhang2014}. 

 The key to Wilf's acceleration method \cite{Wilf1999} has to do with hypergeometric summands that satisfy first-order, holonomic 
 recurrences and that produce sums that satisfy first-order inhomogeneous recurrences. As emphasized in our recent work on Wilf's 
 method \cite{CampbellLevrie2024}, it is only in exceptional cases that Zeilberger's algorithm produces a first-order recurrence of the 
 desired form. 

 Let $F(n, k)$ be a bivariate, hypergeometric function. Let $G(n ,k)$ denote the companion to $F(n, k)$ produced by Zeilberger's 
 algorithm, writing $G(n, k) = R(n, k) F(n, k)$ for a rational function $R(n ,k)$. Suppose that $F(n, k)$ satisfies a first-order, holonomic 
 difference equation, writing 
\begin{equation}\label{requiredholonomic}
 p_{1}(n) F(n + r, k) + p_{2}(n) F(n, k) = G(n, k+1) - G(n, k), 
\end{equation}
 for fixed polynomials $p_{1}$ and $p_{2}$, with $r = 1$ for the purposes of this paper. To apply Wilf's method \cite{Wilf1999}, we are 
 to also require the condition such that 
\begin{equation}\label{requiredvanish}
 \lim_{m \to \infty} G(n, m) = 0. 
\end{equation}
 With this assumption, by writing $f(n) = \sum_{k=0}^{\infty} F(n, k)$, a telescoping argument applied using the difference equation in 
 \eqref{requiredholonomic} produces 
\begin{equation}\label{recursionf}
 f(n) = -\frac{G(n, 0)}{p_{2}(n)} - \frac{ p_{1}(n) }{ p_{2}(n) } f(n + 1), 
\end{equation}
 again letting $r = 1$, and noting the contrast to the Chu--Zhang recurrence in \eqref{Omegarecurrence}. By iteratively applying 
 \eqref{recursionf}, if this produces a convergent summation, then the rate of convergence is determined by the quotient of the 
 leading coefficients of $p_{1}$ and $p_{2}$. We previously applied this approach using functions $F(n, k)$ such that $n$ appears 
 with positive coefficients within the initial arguments of inverted Pochhammer symbols, with input functions such as $F(n, k) = \frac{ 
 (a)_{k} (b)_{k} }{ (n)_{k}^{2} }$ and $F(n, k) = \frac{ (a)_{k} (b)_{k} }{ \left( n \right)_{k} (2n)_{k} }$ for free parameters $a$ and $b$. By 
 using an input function derived by shifting the $k$-argument of a special case of 
\begin{equation}\label{contrastChuZhang}
 \left( n + 2 k \right) 
 \left[ \begin{matrix} 
 b, c, d, e \vspace{1mm} \\ 
 1 + n - b, 1 +n - c, 1 + n - d, 1 + n - e \end{matrix} \right]_{k} 
\end{equation}
 in place of the summand 
\begin{equation}\label{inplaceof}
 \left( n + 2 k \right) 
 \left[ \begin{matrix} 
 b, c, d, e \vspace{1mm} \\ 
 1 + n - b, 1 - n - c, 1 - n - d, 1 + n - e \end{matrix} \right]_{k} 
\end{equation}
 corresponding to $\Omega(n;b,c,d,e) $ (noting the differing signs for the coefficients for $n$ among \eqref{contrastChuZhang} and 
 \eqref{inplaceof}), this has led us to construct infinite families of generalizations of Guillera's formulas in \eqref{secondknown} and 
 \eqref{mainGuillera} inequivalent to the results from Chu and Zhang \cite{ChuZhang2014}. In turn, this has led us to construct and 
 prove series expansions for $\frac{1}{\pi^2}$ that are of the same convergence rates as in \eqref{secondknown} and 
 \eqref{mainGuillera} and that have not previously been known. 

\subsection{Survey}
 A brief survey on known series for $\frac{1}{\pi^2}$ that are hypergeometric and algebraic and that have absolute convergence rates 
 strictly less than $1$ is given below. This emphasizes the originality of our series for $\frac{1}{\pi^2}$ and the motivation 
 surrounding these expansions. 

 In addition to the series for $\frac{1}{\pi^2}$ given above, a notable instance of a series satisfying the given conditions was 
 introduced in 2011 by Guillera \cite{Guillera2011} and is such that 
\begin{equation}\label{Guillera2764}
 \frac{ 48 }{\pi^2} 
 = \sum_{k = 0}^{\infty} 
 \left( \frac{27}{64} \right)^{k} 
 \left[ \begin{matrix} 
 \frac{1}{3}, \frac{1}{2}, \frac{1}{2}, \frac{1}{2}, \frac{2}{3} \vspace{1mm} \\ 
 1, 1, 1, 1, 1 \end{matrix} \right]_{k} 
 \left( 74 k^2 + 27 k + 3 \right). 
\end{equation}
 In addition to the formulas for $\frac{1}{\pi^2}$ in \eqref{mainGuillera} and \eqref{claimedChuZhang} proved by Chu and Zhang in 
 2014 \cite{ChuZhang2014} via $\Omega$-recurrences, such recurrences were also used to prove formulas for $\frac{1}{\pi^2}$ such as 
\begin{align}
 \frac{256}{\pi^2} 
 & = \sum_{k=0}^{\infty} 
 \left( -\frac{1}{4} \right)^{k}
 \left[ \begin{matrix} 
 \frac{1}{2}, \frac{1}{2}, \frac{1}{2}, \frac{3}{2}, \frac{3}{2} \vspace{1mm} \\ 
 1, 2, 2, 2, 2 \end{matrix} \right]_{k} 
 \left( 40k^3 + 108k^2 + 94 k + 27 \right), \label{improvecubic} \\
 \frac{32}{\pi^2} 
 & = \sum_{k=0}^{\infty} 
 \left( \frac{1}{16} \right)^{k} 
 \left[ \begin{matrix} 
 \frac{1}{4}, \frac{1}{2}, \frac{1}{2}, \frac{1}{2}, \frac{3}{4} \vspace{1mm} \\ 
 1, 1, 1, 1, 1 \end{matrix} \right]_{k} 
 \left( 120k^2 + 34k + 3 \right), \nonumber \\ 
 \frac{128}{\pi^2} 
 & = \sum_{k=0}^{\infty} 
 \left( \frac{1}{16} \right)^{k} 
 \left[ \begin{matrix} 
 -\frac{1}{4}, \frac{1}{4}, \frac{1}{2}, \frac{1}{2}, \frac{1}{2} \vspace{1mm} \\ 
 1, 1, 1, 2, 2 \end{matrix} \right]_{k} 
 \left( 120 k^2 + 118k + 13 \right), \nonumber \\ 
 \frac{256}{\pi^2} 
 & = \sum_{k=0}^{\infty} 
 \left( -\frac{1}{27} \right)^{k} 
 \left[ \begin{matrix} 
 \frac{1}{4}, \frac{1}{4}, \frac{1}{4}, \frac{1}{2}, \frac{3}{4}, \frac{3}{4}, \frac{3}{4} \vspace{1mm} \\ 
 1, 1, 1, 1, 1, \frac{4}{3}, \frac{5}{3} \end{matrix} \right]_{k} \times \nonumber \\ 
 & \ \ \ \ \ \ \ \left( 7168 k^4 + 8832 k^3 + 3376 k^2 + 492 k + 27 \right), \nonumber \ \text{and} \\ 
 \frac{256}{\pi^2} 
 & = \sum_{k=0}^{\infty} 
 \left( -\frac{1}{27} \right)^{k} 
 \left[ \begin{matrix} 
 \frac{1}{4}, \frac{1}{4}, \frac{1}{4}, \frac{1}{2}, \frac{3}{4}, \frac{3}{4}, \frac{3}{4} \vspace{1mm} \\ 
 1, 1, 1, 1, 1, \frac{4}{3}, \frac{5}{3} \end{matrix} \right]_{k} \times \nonumber \\
 & \ \ \ \ \ \ \ \left( 7168 k^4 + 8832 k^3 + 3376 k^2 + 492 k +27 \right). \nonumber 
\end{align}
 Since we are to generalize Guillera's series for $\frac{1}{\pi^2}$ of convergence rate $-\frac{1}{1024}$, we highlight the series 
\begin{equation*}
 \frac{2048}{\pi^2} 
 = \sum_{k = 0}^{\infty} 
 \left( - \frac{1}{1024} \right)^{k} 
 \left[ \begin{matrix} 
 -\frac{1}{2}, \frac{1}{2}, \frac{1}{2}, \frac{3}{2}, \frac{3}{2} \vspace{1mm} \\ 
 1, 1, 1, 2, 2 \end{matrix} \right]_{k} 
 \big( 1640 k^3 + 3476 k^2 + 2046 k + 207 \big) 
\end{equation*}
 and 
\begin{equation*}
 \frac{131072}{ 9 \pi^2} 
 = \sum_{k = 0}^{\infty} 
 \left( - \frac{1}{1024} \right)^{k} 
 \left[ \begin{matrix} 
 -\frac{1}{2}, \frac{1}{2}, \frac{1}{2}, \frac{5}{2}, \frac{5}{2} \vspace{1mm} \\ 
 1, 2, 2, 2, 2 \end{matrix} \right]_{k} 
 \big( 1640 k^3 + 4788 k^2 + 4614 k + 1475 \big) 
\end{equation*}
 from Chu and Zhang \cite{ChuZhang2014}, who also introduced series for $\frac{1}{\pi^2}$ of convergence rates $-\frac{16}{27}$ 
 and $\frac{27}{64}$. Recursions for $\Omega$-sums were further applied by Chu \cite{Chu2021Infinite} to reprove such past results. 
 A similar approach involving the Gould--Hsu inverse series relations was applied by Chu \cite{Chu2023} to prove series expansions 
 for $\frac{1}{\pi^2}$ of convergence rate $\frac{1}{16}$, and this includes the formulas such that 
\begin{align*}
 \frac{256}{3\pi^2} 
 & = \sum_{k=0}^{\infty} 
 \left( \frac{1}{16} \right)^{k} 
 \left[ \begin{matrix} 
 -\frac{1}{2}, \frac{1}{4}, \frac{1}{2}, \frac{3}{4}, \frac{3}{2} \vspace{1mm} \\ 
 1, 1, 1, 2, 2 \end{matrix} \right]_{k} 
 \left( 80k^3 + 148 k^2 + 80k + 9 \right) \ \text{and} \\
 \frac{512}{\pi^2} 
 & = \sum_{k=0}^{\infty} 
 \left( \frac{1}{16} \right)^{k} 
 \left[ \begin{matrix} 
 \frac{1}{2}, \frac{1}{2}, \frac{3}{4}, \frac{5}{4}, \frac{3}{2} \vspace{1mm} \\ 
 1, 1, 1, 2, 2 \end{matrix} \right]_{k} 
 \left( 240 k^3 + 532 k^2 + 336 k + 45 \right). 
\end{align*}
 Using a recursion for 
\begin{equation}\label{twofree}
 \left( x + \frac{y - 1}{2} \right) \sum_{k=0}^{\infty} 
 \left[ \begin{matrix} 
 x, x, x, x, x + \frac{y+1}{2}, 1 \vspace{1mm} \\ 
 x+y, x+y, x+y, x+y, x + \frac{y-1}{2} \end{matrix} \right]_{k}, 
\end{equation}
 Levrie and Campbell \cite{LevrieCampbell2022} proved that 
\begin{equation}\label{CLprevious}
 \frac{2^{15}}{ 3^{4} \pi^2 } 
 = \sum_{k = 0}^{\infty} 
 \left( -\frac{1}{4} \right)^{k} 
 \left[ \begin{matrix} 
 \frac{1}{2}, \frac{1}{2}, \frac{1}{2}, \frac{1}{2}, \frac{1}{2} \vspace{1mm} \\ 
 1, 3, 3, 3, 3 \end{matrix} \right]_{k} 
 \left( 20 k^2+56 k+41 \right). 
\end{equation}
 With regard to to two free parameters in \eqref{twofree}, our three-parameter generalization of Guillera's formula in improves upon the 
 acceleration for \eqref{twofree}, and our results as in \eqref{motivating} have not previously been known. Recently, Au \cite{Au2025} 
 introduced a technique relying on \emph{WZ seeds} to prove remarkable results on hypergeometric expansions for powers of $ 
 \frac{1}{\pi}$, and this includes Au's proof of the previously conjectured formula such that $$ \frac{12}{\pi^2} = \sum_{k 
 =0}^{\infty} 
 \left( \frac{4}{27} \right)^{k} 
 \left[ \begin{matrix} 
 \frac{1}{2}, \frac{1}{2}, \frac{1}{2}, \frac{1}{2}, \frac{1}{2}, \frac{1}{2}, \frac{1}{2} \vspace{1mm} \\ 
 \frac{7}{6}, \frac{5}{6}, 1, 1, 1, 1, 1 \end{matrix} \right]_{k} \left( 92 k^3 + 54k^2 + 12 k + 1 \right). $$ Observe that Guillera's series in 
 \eqref{mainGuillera} agrees, up to integer parameter differences of the Pochhammer symbols and up to the polynomial summand factor, 
 with the Chu--Zhang series in both \eqref{claimedChuZhang} and \eqref{improvecubic}. This motivates a full exploration as to how 
 series of this form can be derived, as in Section \ref{sectionMain} below. Also observe that the polynomial summand factor in 
 \eqref{improvecubic} is cubic, in contrast to the quadratic summand factors in both \eqref{mainGuillera} and 
 \eqref{claimedChuZhang}. This motivates how it could be possible to derive infinite families of generalizations of Guillera's formula 
 so as to obtain quadratic summand factors, as in Section \ref{sectionMain}. 

\section{Main results}\label{sectionMain}
 Experimentally, we have discovered that the application of Zeilberger's algorithm to $$ F(n, k) 
 := \left[ \begin{matrix} 
 a, a, a, a \vspace{1mm} \\ 
 1 + n - a, 
 1 + n - a, 1 + n - a, 
 1 + n - a \end{matrix} \right]_{k + b} (n + 2 k + 2 b) $$ yields a recurrence satisfying the required properties according to Wilf's 
 acceleration method \cite{Wilf1999}, noting that $F(n, k)$, as defined above, does not agree with the Chu--Zhang summands of the 
 form indicated in \eqref{inplaceof}. This has led us to prove the acceleration identity given in Theorem \ref{maintheorem} below. 

 As we later demonstrate, the $(a, b, n) = \big(\frac{1}{2}, 0, \frac{3}{2} \big)$ case of Theorem \ref{maintheorem} provides a copy of 
 Guillera's formula in \eqref{mainGuillera}. So, Theorem \ref{maintheorem} provides a three-parameter generalization of Guillera's 
 formula in \eqref{mainGuillera}. 

\begin{theorem}\label{maintheorem}
 Let 
\begin{multline*}
 \mathcal{R}(n, j) := 
 \frac{(n-2 a+1)^5}{(2 n-4 a+1) (2 n-4 a+2 j+1) (n-a+1)^4} \big( 10 a^2- 8 a b- 14 a j- \\ 
 14 a n-6 a+2 b^2+6 b j+6 b n+2 b+5 j^2+10 j n+4 j+5 n^2+4 n+1 \big).
\end{multline*}
 Then, for $F(n, k)$ as specified above, the series $ \sum_{k=0}^{\infty} F(n ,k) $ equals 
\begin{multline*}
 \sum_{j=0}^{\infty} 
 \text{{\footnotesize{ $ 
 \left( -\frac{1}{4} \right)^{j} 
 \left[ \begin{matrix} 
 n-2a+2, n-2a+2, n-2a+2, n-2a+2, n-2a+2 \vspace{1mm} \\ 
 n-2a+\frac{3}{2}, n - a + 2, n - a + 2, n - a + 2, n - a + 2 \end{matrix} \right]_{j-1} \times $ }}} \\ 
 \text{{\footnotesize{ $ \left[ \begin{matrix} 
 a, a, a, a \vspace{1mm} \\ 
 n - a + j + 1, n - a + j + 1, n - a + j + 1, n - a + j + 1 \end{matrix} \right]_{b} 
 \mathcal{R}(n, j) $ }}}
\end{multline*}
 for $a < \frac{2n+1}{4}$. 
\end{theorem}

\begin{proof}
 By applying Zeilberger's algorithm to $F(n, k)$, as specified, we obtain the rational certificate
\begin{multline*}
 R(n, k) = \frac{1}{2 b+2 k+n} 
 (a-n-1)^4 \big(10 a^2-8 a b-8 a k-14 a n - \\ 
 6 a + 2 b^2+4 b k+6 b n+2 b+2 k^2+6 k n+2 k+5 n^2+4 n+1 \big). 
\end{multline*}
 By then setting $ G(n, k) = R(n, k) F(n, k)$, with $ p_{1}(n) = (2 a-n-1)^5 $ and $ p_{2}(n) = 2 (4 a-2 n-1) (a-n-1)^4$, we obtain the 
 difference equation $$ p_{1}(n) F(n+1, k) + p_{2}(n) F(n ,k) = G(n, k+1) - G(n, k). $$ From the given definition for $G(n, k)$, if $a < 
 \frac{2n+1}{4}$, then the limiting relation in \eqref{requiredvanish} holds. In this case, again for 
\begin{equation}\label{fforFspecified}
 f(n) = \sum_{k=0}^{\infty} F(n, k), 
\end{equation}
 the difference equation in \eqref{recursionf} is then seen to hold. So, by writing 
\begin{equation}\label{displayr1r2}
 r_{1}(n) = -\frac{G(n, 0)}{p_{2}(n)} \ \ \ \text{and} \ \ \ r_{2}(n) = -\frac{p_{1}(n)}{p_{2}(n)} 
\end{equation}
 the iterative application of the recursion in \eqref{recursionf} produces the recursion such that $$ f(n) = \sum_{j=-1}^{m} \left( 
 \prod_{i=0}^{j} r_{2}(n+i) \right) r_{1}(n+j+1) + \left( \prod_{i=0}^{m+1} r_{2}(n+i) \right) f(n+m+2) $$ for positive integers $m$. 
 For $a < \frac{2n+1}{4}$, the final term vanishes as $m \to \infty$, and this gives us an equivalent version of the desired result. 
\end{proof}

 The $(a, b, n) = \big(\frac{1}{2}, 0, \frac{3}{2} \big)$ case of Theorem \ref{maintheorem} yields 
\begin{equation*}
 \sum_{k=0}^{\infty} 
 \left[ \begin{matrix} 
 \frac{1}{2}, \frac{1}{2}, \frac{1}{2}, \frac{1}{2} \vspace{1mm} \\ 
 2, 2, 2, 2 \end{matrix} \right]_{k} 
 (4k+3) = 
 \frac{1}{8} \sum_{j=0}^{\infty} 
 \left( -\frac{1}{4} \right)^{j} 
 \left[ \begin{matrix} 
 \frac{3}{2}, \frac{3}{2}, \frac{3}{2}, \frac{3}{2}, \frac{3}{2} \vspace{1mm} \\ 
 2, 2, 2, 2, 2 \end{matrix} \right]_{j} 
 \left( 20 j^2+48 j+29 \right). 
\end{equation*} 
 The partial sums for the left-hand series may be evaluated in closed form, with 
\begin{equation}\label{partialforGuillera}
 \sum_{k=0}^{n} 
 \left[ \begin{matrix} 
 \frac{1}{2}, \frac{1}{2}, \frac{1}{2}, \frac{1}{2} \vspace{1mm} \\ 
 2, 2, 2, 2 \end{matrix} \right]_{k} 
 (4k+3) 
 = 16 - \left[ \begin{matrix} 
 \frac{3}{2}, \frac{3}{2}, \frac{3}{2}, \frac{3}{2} \vspace{1mm} \\ 
 2, 2, 2, 2 \end{matrix} \right]_{n} \left( 8 n^2+20 n+13 \right). 
\end{equation}
 By setting $n \to \infty$ in \eqref{partialforGuillera}, we find that the $(a, b, n) = \big(\frac{1}{2}, 0, \frac{3}{2} \big)$ case of Theorem 
 \ref{maintheorem} yields $$ 16 - \frac{128}{\pi^2} = \frac{1}{8} \sum_{j=0}^{\infty} \left( -\frac{1}{4} \right)^{j} 
 \left[ \begin{matrix} 
 \frac{3}{2}, \frac{3}{2}, \frac{3}{2}, \frac{3}{2}, \frac{3}{2} \vspace{1mm} \\ 
 2, 2, 2, 2, 2 \end{matrix} \right]_{j} \left( 20 j^2+48 j+29 \right), $$ so that a reindexing argument gives us a copy of Guillera's 
 formula in \eqref{mainGuillera}. This leads us to introduce extensions of Guillera's formula and of the Chu--Zhang formula in 
 \eqref{claimedChuZhang} and of Campbell and Levrie's formula in \eqref{CLprevious}. 

\begin{example}
 For $(a, b, n) = \big( \frac{1}{2}, 1, \frac{3}{2} \big)$, 
 we obtain the Chu--Zhang series for $\frac{1}{\pi^2}$ in \eqref{claimedChuZhang}. 
\end{example}

\begin{example}
 For $(a, b, n) = \big( -\frac{1}{2}, 3, \frac{1}{2} \big)$, 
 we obtain Campbell and Levrie's series in 
 \eqref{CLprevious}. 
\end{example}

\begin{example}
 For $(a, b, n) = \big( -\frac{1}{2}, 4, \frac{1}{2} \big)$, 
 we obtain 
\begin{equation*}
 \frac{2^{19}}{ 5^{5} \pi^2 } 
 = \sum_{k = 0}^{\infty} 
 \left( -\frac{1}{4} \right)^{k} 
 \left[ \begin{matrix} 
 \frac{1}{2}, \frac{1}{2}, \frac{1}{2}, \frac{1}{2}, \frac{1}{2} \vspace{1mm} \\ 
 1, 4, 4, 4, 4 \end{matrix} \right]_{k} 
 \left( 4 k^2+16 k+17 \right). 
\end{equation*}
\end{example}

\begin{example}
 For $(a, b, n) = \big( -\frac{1}{2}, 5, \frac{1}{2} \big)$, 
 we obtain 
\begin{equation*}
 \frac{2^{31}}{ 5^4 7^4 \pi^2 }
 = \sum_{k = 0}^{\infty} 
 \left( -\frac{1}{4} \right)^{k} 
 \left[ \begin{matrix} 
 \frac{1}{2}, \frac{1}{2}, \frac{1}{2}, \frac{1}{2}, \frac{1}{2} \vspace{1mm} \\ 
 1, 5, 5, 5, 5 \end{matrix} \right]_{k} 
 \left( 20 k^2+104 k+145 \right). 
\end{equation*}
\end{example}

\begin{example}
 For $(a, b, n) = \big( -\frac{1}{2}, 0, \frac{7}{2} \big)$, 
 we obtain the motivating result highlighted in \eqref{motivating}. 
\end{example}

\begin{example}
 For $(a, b, n) = \big( -\frac{1}{2}, -1, \frac{5}{2} \big)$, 
 we obtain 
\begin{equation*}
 \frac{1024}{15 \pi^2} 
 = \sum_{k = 0}^{\infty} 
 \left( -\frac{1}{4} \right)^{k} 
 \left[ \begin{matrix} 
 \frac{5}{2}, \frac{5}{2}, \frac{5}{2}, \frac{5}{2}, \frac{5}{2} \vspace{1mm} \\ 
 1, 1, 1, 1, 3 \end{matrix} \right]_{k} 
 \left( 4 k^2+8 k+5 \right). 
\end{equation*}
\end{example}

\begin{example}
 For $(a, b, n) = \big( -\frac{1}{2}, -3, \frac{3}{2} \big)$, 
 we obtain 
\begin{equation*}
 -\frac{8192}{5 \pi ^2} 
 = \sum_{k = 0}^{\infty} 
 \left( -\frac{1}{4} \right)^{k} 
 \left[ \begin{matrix} 
 \frac{7}{2}, \frac{7}{2}, \frac{7}{2}, \frac{7}{2}, \frac{7}{2} \vspace{1mm} \\ 
 1, 1, 1, 1, 4 \end{matrix} \right]_{k} 
 \left( 20 k^2+56 k+49 \right). 
\end{equation*}
\end{example}

\begin{example}
 For $(a, b, n) = \big( -\frac{1}{2}, -3, \frac{5}{2} \big)$, we obtain 
\begin{equation*}
 \frac{2^{18}}{35 \pi^2} 
 = \sum_{k = 0}^{\infty} 
 \left( -\frac{1}{4} \right)^{k} 
 \left[ \begin{matrix} 
 \frac{9}{2}, \frac{9}{2}, \frac{9}{2}, \frac{9}{2}, \frac{9}{2} \vspace{1mm} \\ 
 1, 1, 1, 1, 5 \end{matrix} \right]_{k} 
 \left( 20 k^2+72 k+81 \right). 
\end{equation*}
\end{example}

 In addition to the above applications of Theorem \ref{maintheorem} given by the new series for $\frac{1}{\pi^2}$ highlighted above, 
 Theorem \ref{maintheorem} may be considered in relation to an inequivalent generalization of \eqref{mainGuillera} due to Guillera 
 \cite[Identity 8]{Guillera2011} whereby 
\begin{multline*}
 8 a \sum_{k=0}^{\infty} 
 \left[ \begin{matrix} 
 \frac{1}{2}, \frac{1}{2}, \frac{1}{2}, \frac{1}{2} \vspace{1mm} \\ 
 a+1, a+1, a+1, a+1 \end{matrix} \right]_{k} 
 \left( 4 k + 2 a + 1 \right) = \\ 
 \sum_{k=0}^{\infty} \left( -\frac{1}{4} \right)^{k} 
 \left[ \begin{matrix} 
 a + \frac{1}{2}, a + \frac{1}{2}, a + \frac{1}{2}, a + \frac{1}{2}, a + \frac{1}{2} \vspace{1mm} \\ 
 a + 1, a + 1, a + 1, a + 1, a + 1 \end{matrix} \right]_{k} 
 \big( 20 (k+a)^{2} + 8 (k + a) + 1 \big). 
\end{multline*}

\subsection{A double acceleration method}
 We rewrite \eqref{fforFspecified} as 
\begin{equation*}
 f(n, b) = \sum_{k=0}^{\infty} 
 \left[ \begin{matrix} 
 a, a, a, a \vspace{1mm} \\ 
 n - a + 1, n - a + 1, n - a + 1, n - a + 1 
 \end{matrix} \right]_{k + b} \big( n + 2 k + 2b \big), 
\end{equation*}
 taking $a$ as a parameter. 
 By determining a ``double'' recurrence for $f(n+1,b+1)$ (with respect to both arguments) 
 in terms of $f(n,b)$, as below, 
 the iterative application of our recurrence of this form
 yields a generalization of Guillera's formula in \eqref{secondknown}. 

\begin{theorem}\label{theoremdouble}
 Let $q_{1}(j) = -10 a^2+8 a b+22 a j+14 a n+28 a-2 b^2-10 b j-6 b n-12 b-13 j^2-16 j n-32 j-5 n^2-20 n-20$
 and $q_{2}(j) = 4 a-2 j-2 n-3$, and let 
 $$ s_{1}(n) = \frac{(-2 a+n+1)^5}{4 (4 a-2 n-1) (-a+n+1)^4} $$
 and
 $$ s_{2}(j) = \frac{(-2 a+j+n+2)^5 (2 b+3 j+n+4)}{(-a+b+2 j+n+3)^4}. $$ 
 Then 
\begin{multline*}
 f(n, b) = s_{1}(n) \sum_{j=-1}^{\infty} 
 \left( -\frac{1}{4} \right)^{j} 
 \frac{ q_{1}(j) + s_{2}(j) }{q_{2}(j)} \times \\ 
 \text{{\footnotesize{ $ 
 \left[ \begin{matrix} 
 n-2a + 2, n-2a + 2, n-2a + 2, n-2a + 2, n-2a + 2 \vspace{1mm} \\ 
 n - 2 a + \frac{3}{2}, 
 n - a + 2, n - a + 2, n - a + 2, n - a + 2 
 \end{matrix} \right]_{j} \left[ \begin{matrix} 
 a, a, a, a \vspace{1mm} \\ 
 n + j - a + 2 
 \end{matrix} \right]_{j + b + 1} $ }}} 
\end{multline*}
 for $a < \frac{2n+1}{4}$. 
\end{theorem}

\begin{proof}
 Using Zeilberger's algorithm with respect to the latter argument of $f(n, b)$, we obtain the recursion such that $$ f(n, b+1) -f(n, b) 
 = -\left[ \begin{matrix} 
 a, a, a, a \vspace{1mm} \\ 
 n - a + 1, n - a + 1, n - a + 1, n - a + 1 
 \end{matrix} \right]_{b} (n+2b). $$ Rewriting the functions $r_{1}(n)$ and $r_{2}(n)$ in \eqref{displayr1r2} 
 as $r_{1}(n, b)$ and $r_{2}(n, b)$, respectively, 
 we also write 
 $$ r_{3}(n, b) = -\left[ \begin{matrix} 
 a, a, a, a \vspace{1mm} \\ 
 n - a + 1, n - a + 1, n - a + 1, n - a + 1 
 \end{matrix} \right]_{b} (n+2b). $$
 By then writing 
 $ r_{4}(n ,b) = r_{1}(n, b) + r_{2}(n, b) r_{3}(n+1, b), $
 the given recursions for $f(n, b)$ together yield 
\begin{equation}\label{togetheryield}
 f(n, b) = r_{4}(n, b) + r_{2}(n, b) f(n+1, b+1). 
\end{equation}
 The iterative application of \eqref{togetheryield} yields the recurrence 
\begin{multline*}
 f(n, b) 
 = \sum_{j=-1}^{m} \left( \prod_{i=0}^{j} r_{2}(n+i, b+i) \right) 
 r_{4}(n+j+1,b+j+1) + \\
 \left( \prod_{i=0}^{m+1} r_{2}(n+i, b+i) \right) f(n+m+2, b+m+2) 
\end{multline*}
 for positive integers $m$. Setting $m \to \infty$, the latter term vanishes, providing an equivalent version of the desired result. 
\end{proof}

 For $(a, b, n) = \big( \frac{1}{2}, 1, \frac{3}{2} \big)$, we obtain a copy of the Guillera formula in \eqref{secondknown}, so that 
 Theorem \ref{theoremdouble} provides a three-parameter generalization of Guillera's formula of convergence rate $-\frac{1}{1024}$. 
 This motivates our new series for $\frac{1}{\pi^2}$ of the same convergence rate, as below. 

\begin{example}
 Setting $ (a, b, n) = \big( \frac{1}{2}, 0, \frac{3}{2} \big)$ in Theorem \ref{theoremdouble}, we obtain the moviating result highlighted 
 in \eqref{motivating2}. 
\end{example}

\begin{example}
 Setting $ (a, b, n) = \big( \frac{1}{2}, 0, \frac{5}{2} \big) $ in Theorem \ref{theoremdouble}, we obtain that 
\begin{multline*}
 -\frac{1024}{\pi^2} 
 = \sum_{k=0}^{\infty} 
 \left( -\frac{1}{1024} \right)^{k} 
 \left[ \begin{matrix} 
 -\frac{1}{2}, -\frac{1}{2}, -\frac{1}{2}, -\frac{1}{2}, \frac{3}{2} \vspace{1mm} \\ 
 1, 1, 1, 1, 2 
 \end{matrix} \right]_{k} \times \\
 \big( 13120 j^6+34368 j^5+36144 j^4+17888 j^3+2956 j^2-540 j-99 \big). 
\end{multline*}
\end{example}

\begin{example}
 Setting $ (a, b, n) = \big( \frac{1}{2}, 0, \frac{7}{2} \big) $ 
 in Theorem \ref{theoremdouble}, we obtain that 
\begin{multline*}
 -\frac{64}{\pi ^2} 
 = \sum_{k=0}^{\infty} 
 \left( -\frac{1}{1024} \right)^{k} 
 \left[ \begin{matrix} 
 \frac{3}{2}, \frac{3}{2}, \frac{3}{2}, \frac{3}{2}, \frac{3}{2} \vspace{1mm} \\ 
 1, 1, 1, 1, 2 
 \end{matrix} \right]_{k} \times \\
 \frac{13120 j^6-2368 j^5-2320 j^4-4320 j^3-1620 j^2+972 j+729}{(2 j-3)^4 (2 j-1)^4}. 
\end{multline*}
\end{example}

 The $(a, b, n) = \big( 1, 0 ,2 \big)$ case of Theorem \ref{theoremdouble} gives us the formula $$ 64 \zeta(3) = \sum_{k = 
 0}^{\infty} 
 \left( -\frac{1}{1024} \right)^{k} 
 \left[ \begin{matrix} 
 1, 1, 1, 1, 1 \vspace{1mm} \\ 
 \frac{3}{2}, \frac{3}{2}, \frac{3}{2}, \frac{3}{2}, \frac{3}{2} 
 \end{matrix} \right]_{k} \big( 205 k^2 + 250 k + 77 \big)$$ for Ap\'{e}ry's constant $\zeta(3) = \frac{1}{1^3} + \frac{1}{2^3} + \cdots$ 
 due to Ambeberhan and Zeilberger \cite{AmdeberhanZeilberger1997}. Similarly, the $(a, b, n) = \big( \frac{1}{2}, 0, 1 \big)$ case of 
 Theorem \ref{theoremdouble} gives us a formula for $\zeta(3)$ given by Chu and Zhang \cite[Example 62]{ChuZhang2014}. This 
 motivates the application of Theorem \ref{theoremdouble} and variants and extensions of Theorem \ref{theoremdouble} in 
 the derivation of new series for constants other that $\frac{1}{\pi^2}$. For example, the $(a, b, n) = \big( 1, 1, \frac{5}{2} \big)$ case 
 of Theorem \ref{theoremdouble} yields 
\begin{multline*}
 \frac{ 81 \pi^2 }{16} 
 = \sum_{k=0}^{\infty} 
 \left( -\frac{1}{1024} \right)^{k} 
 \left[ \begin{matrix} 
 \frac{1}{2}, \frac{1}{2}, \frac{1}{2}, \frac{1}{2}, \frac{1}{2}, 1, 1, 1 \vspace{1mm} \\ 
 \frac{5}{4}, \frac{5}{4}, \frac{5}{4}, \frac{5}{4}, \frac{7}{4}, \frac{7}{4}, \frac{7}{4}, \frac{7}{4} 
 \end{matrix} \right]_{k} \times \\
 \big( 1640 k^6+5936 k^5+8738 k^4+6664 k^3+2762 k^2+587 k+50 \big). 
\end{multline*}

\subsection*{Acknowledgements}
   The author is grateful to acknowledge support from a Killam Postdoctoral Fellowship    from the Killam Trusts. The author thanks Paul   
  Levrie for useful feedback.

 \

{\textsc{John M. Campbell}} 

\vspace{0.1in}

Department of Mathematics and Statistics

Dalhousie University

6299 South St, Halifax, NS B3H 4R2

\vspace{0.1in}

{\tt jh241966@dal.ca}

\end{document}